\documentclass[11pt,leqno]{article}
\usepackage{graphicx}
\usepackage{wrapfig}
\usepackage{amsfonts, amsthm}
\usepackage{amsmath}
\usepackage{xcolor}
\begin{document}

\numberwithin{equation}{section}
\newtheorem{prop}{Proposition}[section]
\newtheorem{exa}{Example}[section]
\newtheorem{theo}{Theorem}[section]
\newtheorem{rem}{Remark}[section]
\newtheorem{lem}{Lemma}[section]
\newtheorem{defi}{Definition}[section]
\newtheorem{coro}{Corollary}[section]
\newtheorem{prob}{Problem}[section]
\newtheorem{nota}{Notation}[section]
\newtheorem{cond}{Conditions}[section]
\newtheorem{pot}{Proof of Theorem}[section]
\def\proj{{\Bbb P}^2} 
\def\R{{\Bbb R}}
\def\Z{{\Bbb Z}}
\def\C{{\Bbb C}}
\def\Q{{\Bbb Q}}
\newcommand{\ffc}[1]{{{\cal F}}(#1)}
\newcommand{\wc}[1]{{{\cal M}}(#1)}
\newcommand{\ps}[1]{{{\Bbb P}}^{#1}}
\def\F{{\cal F}}
\def\sM{{\sf M}}
\def\P{{\cal P}}
\def\Pt{\C[t]}
\def\cT{{\cal T}}
\def\RR{{\cal R}} 
\def\li{{\Bbb L}}
\def\L{{\cal L}}

\def\mC{\mathcal {C}}
\def\mF{\mathcal {F}}
\def\mM{\mathcal {M}}
\def\mR{\mathcal {R}}
\newcommand{\cfc}[1]{\{\delta_t\}_{t\in #1}}
\begin{center}
{\LARGE\bf Brieskorn module and  Center conditions: pull-back of differential equations in projective space
\footnote{
 This work is partially supported by T\"UB\.{I}TAK project 116F130 "Period integrals associated to algebraic varieties.  "
}

\vspace{.25in} {\large {\sc Yadollah Zare\  ,\ Susumu Tanab\'e}}
\\}

\end{center}
\begin{abstract}
The moduli space of algebraic foliations on $\mathbb{P}^2$ of a fixed degree and with a center
singularity has many irreducible components. We find a basis of the Brieskorn module defined for a  rational function and 
prove that pull-back foliations forms an irreducible component of the moduli space. The main tools are Picard-Lefschetz theory of a rational function
in two variables,  period integrals and Brieskorn module. 
 \end{abstract}
\setcounter{section}{-1}
\section{Introduction}\label{intro}
A holomorphic foliation $\mathcal{F}(\alpha)$ in $\mathbb{P}^2$ is defined by a 
1-form $\alpha=AdX+BdY+CdZ$, where $A,B$ and $C$ are homogeneous polynomials of degree $d+1$ satisfying the identity $\alpha(\upsilon)=XA+YB+ZC = 0$ for $\upsilon= X\frac{\partial}{\partial X}+Y\frac{\partial}{\partial Y}+Z\frac{\partial}{\partial Z}$, the Euler vector field.
 For  generic $A,B$ and $C$  the degree of the foliation $\mathcal{F}(\alpha)$ is defined as  $d = deg(A)-1$ (see \cite{A.L} and \cite{CL}). 
The space of algebraic foliations
\[
\mathcal{F}(\alpha)\ \ ,\ \ \  \alpha\in \Omega_{d+1}^1 ,
\]   
where 
\begin{equation}
\Omega_{d+1}^1:=\{\alpha=AdX+BdY+CdZ\ | \ A,B,C\in \mathbb{C}[X,Y,Z]_{ d+1}^h,\ \alpha(\upsilon)\equiv 0\},
\label{Omegad+1}
\end{equation}
is the projectivization of the vector space $\Omega_{d+1}^1$ and is denoted by $\mathcal{F}(2,d)$. From here on we use the notation $\mathbb{C}[x,y,z]_{\sf D_0}^h$ to denote the ring of homogeneous polynomials of degree $\sf D_0.$ The space $\mathcal{F}(2,d)$ is a rational variety $\mathbb{C}^N$ for some $N$, the space of coefficients of  polynomials  $A,B,C$. A \emph{singularity}
of $\mathcal{F}(\alpha)$ is a common zero of $A,B$ and $C$. We denote the set of singularities of $\mathcal{F}(\alpha)$  by $Sing(\mathcal{F}(\alpha))$. For an isolated singularity $p\in Sing(\mathcal{F}(\alpha))$, if there is a holomorphic coordinate system $(\tilde{x},\tilde{y})$ in a neighborhood of the point $p$ with $\tilde{x}(p)=0$, $\tilde{y}(p)=0$ such that in this coordinate system
$$
\alpha|\wedge d(\tilde{x}^2+\tilde{y}^2)=0,
$$ 
holds then the point $p$ is called a \emph{center singularity}. 
The closure of the set of algebraic foliations of fixed degree $d$ with at least one center in $\mathcal{F}(2,d)$, which is denoted by $\mathcal{M}(2,d)$, is an algebraic subset of $\mathcal{F}(2,d)$ (see for instance, \cite{Mo1} and \cite{A.L}). Identifying irreducible components of $\mathcal{M}(2,d)$ is the center condition problem in the context of polynomial differential equations on the real plane.
The complete classification of irreducible components of $\mathcal{M}(2,2)$ is done by 
H. Dulac in \cite{Du} (see also \cite[p.601]{CL}). For the classification of polynomial 1-forms of degree 3 the reader can consult \.{Z}o{\l}{\c{a}}dek's articles \cite{Z-1}, \cite{Z-2}. This classification gives applications on the number of limit cycles in the context of polynomial differential equations on the real plane.

Consider the morphism  
\begin{equation} 
\begin{array}{ll}
\ \  \ \ \  F:\mathbb{P}^2\rightarrow\mathbb{P}^2\\
{[x,y,z]}\rightarrow{[R,S,T]}
\end{array} 
\label{F}
\end{equation}
where $R,S,T \in \mathbb{C}[x,y,z]_{s}^h.$\\
Let  $\mathcal{P}(2,a,s)$ be the set of foliations
\begin{equation}\label{Falpha}
\mathcal{F}(F^*(\alpha)) \ \ \   where \ \ \  \alpha\in\Omega_{a+1}^1 \ . 
\end{equation}

Let us denote by $D$ the zero locus of the determinant of Jacobian matrix $J_F$ of $F$ i.e. 
\begin{equation}
 J_F (x,y,z)= 
\left(\begin{array}{lll}
R_x&R_y&R_z\\
S_x&S_y&S_z\\
T_x&T_y&T_z
\end{array}\right)
\label{JF}
\end{equation}

\begin{equation}
D=\{  [x:y:z]|  det(J_F (x,y,z)) =0\} \subset \mathbb{P}^2.
\label{D}
\end{equation}

\begin{defi}
For a generic morphism $F$ and foliation $\mathcal{F}$, there exists local chart $(\phi,U)$ (resp. $(\psi,V)$)   of the critical point $p\in D$ (resp. $F(p)$)   such that $F|U=(x^2,y)$ and the leaves of $\mathcal{F}|(\psi,V)$ are given by $X+Y^2=t$. Therefore, $F^{\ast}(\mathcal{F})$ with local expression $x^2+y^2=t$ has a center singularity. We call this singularity of the foliation $F^{\ast}(\mathcal{F})$ \emph{tangency center singularity}.
\label{TCS}
\end{defi}
We define the space $\mathcal{P}(2,a,s)$ of pull-back foliations $\mathcal{F}(\omega)$ defined by
\begin{equation}
  \omega=F^{\ast}(A)dR+F^{\ast}(B)dS+F^{\ast}(C)dT,
\label{AdRBdSCdT}
\end{equation}
where
\[
R,S,T\in\mathbb{C}[x,y,z]_{ s}^h \ , \ A,B,C \in \mathbb{C}[X,Y,Z]_{ a+1}^h
\ ,
\]
\[ RF^{\ast}(A)+SF^{\ast}(B)+T F^{\ast}(C) = 0.\]

\begin{theo}\label{1:28}
The space $\mathcal{P}(2,a,s)$ constitutes an irreducible component of $\mathcal{M}(2, d )$ 
 for $d=  s(a+2)-2,  s \geq 2.$
\end{theo}

This article treats the question of the vector tangent to $\mathcal{M}(2, d)$  at  a foliation with rational first integral
$f.$ This question has been studied for the case of  pull-back of polynomial differential equations in \cite{YZ}.

A  sketch of contents of the  article is as follows.
 In section \ref{pullbackprojective} we prepare notations and notions that will be used in the course of exposition. 
Since the calculation of the tangent space of $\mathcal{M}(2,d)$ at its generic point requires more information, we adopt a strategy to choose a special point $\mathcal{F}_0$ in the intersection of $\mathcal{P}(2,a,s)$  with the space of projective logarithmic foliations \eqref{IPQ}.  Let $\mathcal{F}_{\epsilon }$ be a deformation of $\mathcal{F}_0$ such that for any $\epsilon$, $\mathcal{F}_\epsilon$ has a center singularity closed to a tangency center singularity of $\mathcal{F}_0$. We  assume that $\mathcal{F}_0$ (resp. $\mathcal{F}_{\epsilon}$ ) is defined by $F^*(\alpha_0)$ (resp. $\omega_\epsilon=F^*(\alpha_0)+\epsilon^{1} \omega_1+\cdots $) for $\alpha_0$ \eqref{alpha0}. 
By using the results concerning vanishing cycles from section \ref{PL}  and relatively exact 1-forms from section \ref{REx}, we establish that there are 1-forms $\alpha$ and $ \omega_e$ such that $\omega_1$ be expressed as $\omega_1=F^*(\alpha)+ \omega_e$.
In section \ref{TC} by using the first order Melnikov function and period integral we calculate an explicit form of the tangent vector to the irreducible component ${\mathcal X}(a,s)$ of $\mathcal{M}(2,d)$ containing $\mathcal{P}(2,a,s)$. This calculation gives a proof of Theorem  \ref{8:34:21}.
In section \ref{PMBM} we find a basis for the Brieskon module $H_f$ with  $f=\frac{P^q}{Q^p}$ for $P,Q$ generic polynomials and $p,q$ co-prime. This is used to prove that  $F^{\ast}(B)$, for $B$ a basis of $H_f$, is extendable to a basis of $H_{F^\ast( f)}$.
This result furnishes us with the proof of Theorem
\ref{14:06} that is necessary to establish Corollary \ref{Hfinj} that in its turn  plays an essential r\^ole during the proof of  the key Lemma  \ref{19:03}.

Authors are grateful to Hossein Movasati for useful remarks.
\section{Pull-back of projective foliations}\label{pullbackprojective}

Let $F$ be the morphism defined in \eqref{F}. 
If the coefficients of the pull-back $F^*(\alpha)$ of $\alpha=AdX+BdY+CdZ$ for $A,B,C\in\mathbb{C}[X,Y,Z]_{a+1}^h$  do not have a common factor then the degree of the foliation $F^*(\mathcal{F}(\alpha))$ is $s(a+2)-2$.  

By taking the coefficient of polynomials as coordinates of the map from the space of polynomials to projective space we can see that $\mathcal{P}(2,a,s)$ is an irreducible algebraic subset of $\mathcal{M}(2,d)$. We take a point $\mF(F^*(\alpha))$ in 
$\mathcal{P}(2,a,s)$ for a 1-form $\alpha$ like in \eqref{Omegad+1},  make a deformation $\mathcal{F}_{\epsilon}\in\mathcal{P}(2,a,s)$
and calculate the tangent vector space of $\mathcal{P}(2,a,s)$ at $\mF(F^*(\alpha))$: 
\begin{equation}\label{10:16}
\begin{split}
F_{\epsilon}^*(\alpha_{\epsilon})&=F_{\epsilon}^*(\alpha+\epsilon\alpha_1)+O(\epsilon^2)
\end{split}
\end{equation}
for $\alpha_\epsilon = \alpha + \sum_{j \geq 1} \epsilon^j \alpha_j$ with $\alpha_j \in \Omega^1_{\leq d+1}, \forall j \geq 1.$
The pull-back by a morphism $F_\epsilon= F+ \epsilon F_1 + O(\epsilon^2)$ with $F=(R,S,T)$, $F_1= ( R_1,S_1 , T_1)$ of the  form \eqref{10:16} 
\begin{equation}\label{11:03:13}
F_{\epsilon}^*(\alpha+\epsilon\alpha_1+\cdots)=F^*(\alpha)+\epsilon  \omega_W+O(\epsilon^2).
\end{equation}
Here the 1-form $\omega_W$ has the following form:
\begin{equation}\label{16:48}
\begin{split}
\omega_W&=F^\ast(A)dR_1+F^\ast(B)dS_1+F^\ast(C)dT_1\\
&+\left(R_1.F^\ast(A_X)+S_1.F^\ast(A_Y)+T_1.F^\ast(A_Z)\right)dR\\
&+\left(R_1.F^\ast(B_X)+S_1.F^\ast(B_Y)+T_1.F^\ast(B_Z)\right)dS\\
&+\left(R_1.F^\ast(C_X)+S_1.F^\ast(C_Y)+T_1.F^\ast(C_Z)\right)dT+F^{*}(\alpha_1).
\end{split}
\end{equation}

 Consider a foliation in $\mathbb{P}^{2} $ with the first integral 
\begin{equation}\label{PQfirstintegral}
\begin{split}
f:\mathbb{P}^2&\setminus(\{P=0\}\cap\{Q=0\})\rightarrow\bar{\mathbb{C}}\\
f(X,Y,Z)&=\frac{P(X,Y,Z)^q}{Q(X,Y,Z)^p}
\end{split}
\end{equation}
with $\frac{deg(P)}{deg(Q)}=\frac{p}{q}$, $g.c.d(p,q)=1.$ 
We denote by 
\begin{equation}\label{IPQ}
\mathcal{I}(deg(P)-1, deg(Q)-1)
\end{equation}
the space of projective foliations $\mathcal{F} (\alpha_0)$ with the first integrals of type (\ref{PQfirstintegral}) where
\begin{equation}\label{alpha0}
\alpha_0= qQdP-pPdQ.
\end{equation}
In \cite[Theorem 5.1]{Mo1} it is shown that the space of foliations with a first integral of type (\ref{PQfirstintegral})
is an irreducible component of $\mathcal{M}(2, deg(P)+deg(Q)-2)$ for $deg(P)+deg(Q)-2\geq 2$. This is a generalization of the result \cite{Ily} established for the polynomial first integral  to the case of foliations in  $\mathbb{P}^2$ with a generic rational first integral of type  (\ref{PQfirstintegral}).

 \begin{rem}\label{hM}
By calculating the tangent vector $\omega_W$ to $\mathcal{P}(2,a,n)$ at the point $\mathcal{F}_0=F^{*}(\alpha_0)$  for $\alpha_0$ \eqref{alpha0}, we have $A=qQP_X-pPQ_X$, $B=qQP_Y-pPQ_Y$, $C=qQP_Z-pPQ_Z$ in \eqref{16:48} and obtain
\begin{equation}\label{omegaW}
\omega_W=\omega_{pl}+F^{*}(\alpha_1),
\end{equation}
where
\begin{equation*}
\begin{split}
\omega_{pl}=qF^{*}(Q)dP_1&-pP_1dF^{*}(Q)- qQ_1dF^{*}(P)+pF^{*}(P)dQ_1,\\
P_1&=R_1F^{*}(P_X)+S_1F^{*}(P_Y)+T_1F^{*}(P_Z),\\ Q_1&=R_1F^{*}(Q_X)+S_1F^{*}(Q_Y)+T_1F^{*}(Q_Z).
\end{split}
\end{equation*}

See the calculation of \eqref{3:32}. 
\end{rem}

Let $\mathcal{X}(a,s)$ be the irreducible component of $\mathcal{M}(2,d)$ containing 
$\mathcal{P}(2,a,s)$. To calculate tangent cone of $\mathcal{X}(a,s)$ at a special point in  $ \mathcal{I}(ms-1,ns-1)\cap\mathcal{P}(2,a,s)$. 
Let $F$ be a generic morphism of $\mathbb{P}^2$ into itself such that each component of $F$ and  $f$ be a rational function $f=\frac{P^q}{Q^p}$ where $P,Q$ are two homogeneous polynomials degrees $m,n$ with 
\begin{equation}\label{nma}
a=n+m-2
\end{equation}  i.e. $mq=np$ with the condition \eqref{cond}.
\begin{theo}\label{8:34:21}
Tangent cone of $\mathcal{X}(a,s)$ at the point $\mathcal{F}_0:=\mathcal{F}(F^{*}(\alpha_0))$  for $\alpha_0$ \eqref{alpha0}
 is equal to tangent cone of $\mathcal{P}(2,a,s)$ at this point
\[TC_{\mathcal{F}_0}\mathcal{X}(a,s)=TC_{\mathcal{F}_0}\mathcal{P}(2,a,s).\] 
\end{theo}

Let us consider a deformation $\mathcal{F}(\omega_{\epsilon})\in \mathcal{X}(a,s)$ defined  by the following 1-form:  
\begin{equation}\label{11:31}
\omega_{\epsilon} = F^*(\alpha_0)+\epsilon \omega_1+\epsilon^{2}\omega_{2}+\cdots,\ \ \ \ \ \ \ deg(\omega_j)\leq d+1\ \forall j.
\end{equation} 




The hypothesis $\mathcal{F}(\omega_{\epsilon}) \in \mathcal{X}(a,s)$  implies that it always has a tangency center singularity (Definition \ref{TCS}) near the tangency center singularity $p$ of $\mF(\omega_0)$ defined by $\omega_0 = F^\ast(\alpha_0 ).$ 
 We call it also  {\it tangent critical point } of the rational mapping
$ F^\ast(f): \mathbb{P}^2 \rightarrow \mathbb{P}^2$ in view of the circumstance that requires analysis of vanishing cycles and their monodromy associated to $F^\ast(f)$ in \S  \ref{PL}.

Let us consider the affine coordinate  $E$,  and let $\delta_t$ be a continuous family of vanishing cycles in $(F^\ast(f))^{-1}(t)$ around tangent critical point $p$  and 
$\Sigma\cong \mathbb{C}$ be a transverse section of $\mathcal{F}$ at some points of $\delta_t$. We write Taylor expansion of the deformed holonomy $h_{\epsilon}(t)$ 
\[
h_{\epsilon}(t)-t=M_1(t)\epsilon+M_2(t)\epsilon^2+\cdots+M_i(t)\epsilon^i+\cdots 
\]
where $M_i(t)$ is $i-th$ Melnikov function of the deformation (see \cite[Theorem 1.1]{Fr}, \cite[Section 3]{Mo1}).
 If $\Sigma$ is parametrized by the image of $F^\ast(f)$ i.e. $t= F^\ast(f)(\sigma_0) $  for $\sigma_0 \in\Sigma$  then 
\begin{equation} \label{Mk}
M_1(t)= -\int_{\delta_t}\frac{F^{\ast}(f)\omega_1}{F^{\ast}(P Q)} = - t \int_{\delta_t}\frac{\omega_1}{ F^{\ast}(P Q)}.
\end{equation} 
for $\delta_t \in H_1 ( (F^\ast(f)^{-1}(t), \Z),$ the vanishing cycle associated to the tangent critical point $p \in C_{F^\ast(f)}.$
In fact $M_j(t), \forall j \geq 1$ also vanishes for all $t$ near to zero as $F^*(\mathcal{F}(\omega_\eta ))$ has a tangency center singularity
$\forall \eta \in [0, 2\epsilon].$  


The condition $M_1(t) =0$ plays a central r\^ole in the proof of Theorem \ref{8:34:21}. See Lemma 
\ref{19:03}, Lemma \ref{correctionterm}.

\section{Monodromy action on tangency cycles }\label{PL}

In this section we formulate Theorem \ref{thmFast} that establishes a relation between vanishing cycles associated to $f$ and $F^\ast(f).$
Let  $\mathcal{F}$ be a foliation in $\mathbb{P}^2$ with the first integral $f = \frac{P^q}{Q^p}$ like in (\ref{PQfirstintegral}).
In addition to (\ref{PQfirstintegral}) we impose on $\mathcal{F}$ satisfies the following conditions:

\begin{cond}
(1) The curves $V(P)\subset \proj$ and $V(Q) \subset \proj$ are smooth and intersect each other transversally, and also each one has transversal intersection with line at infinity. We denote the set $V(P) \cap V(Q)$  by $\mathcal{R}.$
(2) All critical points of $f= \frac{P^q}{Q^p}$ are distinct and belong to the affine plane $\mathbb{C}^2$.
(3) $deg (P) >  deg(Q) \geq 2, i.e. m >  n \geq 2.$

\label{cond}
\end{cond}

Let us denote by
$$ \sM_ {f} : \pi_1 (\C \setminus C_{f}, B) \longrightarrow GL(H_1(f^{-1}(B),\mathbb{Z}), \Z)$$
the monodromy representation of the fundamental group of the complement to the critical value set $C_{f}$ of $f$  acting on the 
first homology group of a smooth generic fiber $H_1(f^{-1}(B),\mathbb{Z}).$

In a similar manner we consider the following monodromy representation for the first homology group of a smooth generic fiber
$H_1((F^{\ast}(f))^{-1}(b),\mathbb{Z}):$
$$ \sM_ {F^{\ast}(f)} : \pi_1 (\C \setminus C_{ F^\ast(f)}, b) \longrightarrow GL(H_1((F^{\ast}(f))^{-1}(b),\mathbb{Z}), \Z).$$
By virtue of the assumption made on the critical points of $f$ and the genericity of $F$ the monodromy representations  $\sM_ {f}$ and  $\sM_ {F^{\ast}(f)}$ can be realized with integer coefficients.  Here $C_{F^{\ast}(f)}$ denotes the critical value set of $F^{\ast}(f).$ 

Let us call vanishing cycle $\delta_t$ around a tangent critical point {\it tangency vanishing cycle} (see   (\ref{tangencycycles}) below).

\begin{theo}\label{thmFast}
The morphism  
\begin{equation}
F_*:H_1((F^{\ast}(f))^{-1}(b),\mathbb{Z})\rightarrow H_1(f^{-1}(b),\mathbb{Z}) 
\label{Fast}
\end{equation}
 is surjective and $Ker(F_*)$ is generated by the result of the monodromy group action \\ $\sM_ {F^{\ast}(f)} (\pi_1(\mathbb{C}\setminus C_{F^{\ast}(f)},b))$ on a tangency cycle $\delta_t$ around a tangent critical point. 
\end{theo}    
\begin{proof}
For $D$ defined in (\ref{D}) the morphism $F_|:\mathbb{C}^2\setminus D\to\mathbb{C}^2\setminus F(D)$  is a covering map. Inverse image of each vanishing cycle $\Delta\in H_1(f^{-1}(B),\mathbb{Z})$ via $F_|^{-1}$ contains $s^2$ disjoints vanishing cycles. Moreover, $H_1((F^{\ast}(f))^{-1}(b),\mathbb{Z})$ contains a subgroup consisiting of $s^2$ copies of a group isomorphic to $H_1(f^{-1}(B),\mathbb{Z})$.

Let us take a pull-back vanishing cycle $\delta_1$ such that $F_*(\delta_1)=\Delta.$  According to \cite[ (7.3.5)]{La}, \cite[Theorem 2.3]{Mo5} the action of the monodromy group $\sM_ {F^{\ast}(f)} \left( \pi_1(\mathbb{C}\setminus C_{F^{\ast}(f)},b) \right)$ on the vanishing cycle $\delta_1$  generates a subgroup of $H_1((F^{\ast}(f))^{-1}(b),\mathbb{Z})$ which is isomorphic to $H_1(f^{-1}(B),\mathbb{Z})$.

 It is well-known that Dynkin diagram of $F^{\ast}(f)$ is connected, see for instance \cite{A-M}. If we remove the vertices which correspond to vanishing cycles around the tangent critical points in $D$ (we call them {\it tangency vertices}) the Dynkin diagram becomes $s^2$ disjoint graphs $P_i$, $1\leq i\leq s^2$ each of which is isomorphic to the Dynkin diagram of $f$. See
 \cite[Figure 9]{YZ}.

There exists a local chart around the tangent critical point $p\in D$ such that in this chart $F=(x^2,y)$. Therefore the graph $P_i$ is connected to a graph $P_j$ by tangency vertices. Indeed, a tangency vertex corresponding to a tangency cycle around a tangent critical point $p\in D$ is connected to some $P_i$ or to another vertex corresponding to a vanishing cycle around $p\in D$. Therefore 
we conclude that the monodromy group $\sM_ {F^{\ast}(f)} (\pi_1 (\C \setminus C_{F^{\ast}(f)}, b))$ actions  on the tangency cycle $\delta_t$ generate cycles of the following two types.  The difference between two vanishing cycles associated to different critical values  $c_i\not = c_j $ of $F^\ast(f)$ satisfying $F(c_i) = F(c_j) =c,$ $c\in C_{F^{\ast}(f)}.$
 \begin{equation}
\delta_c^i-\delta_c^j\ ,  \ 1\leq i,j\leq s^2 , \ \ .
    \label{deltac}
\end{equation}
and vanishing cycles that are associated to $p \in D,$ 
 \begin{equation}
\ \ \delta_p ,  \;\;\;\;\ p\in D.
    \label{tangencycycles}
\end{equation}
In \cite[Theorem 4.9, Figure 9]{YZ} cycles (\ref{tangencycycles}) are divided into {\it tangency} and {\it exceptional  vanishing cycles}.

We denote by $H$ the free Abelian group generated by cycles  (\ref{deltac}), (\ref{tangencycycles}) that is a subgroup of 
  $ker(F_*).$

The space of cycles of type (\ref{deltac}) has rank $ (s^2 -1) \mu_f$ for $\mu_f ={\rm rank}\; H_1(f^{-1}(b), \mathbb{Z}).$
More precisely we calculate $\mu_f = (m+n-1) - mn$ as $|V(P) \cap V(Q)|= mn$ (see Proposition \ref{globalMilnor}).  In a similar manner the equality
$ \mu = {\rm rank} \;H_1((F^{\ast}(f))^{-1}(b),\mathbb{Z})  = (s(m+n)-1)^2 - s^2 mn$ holds.

Thus the rank of the space $H$ 
is equal to 
 \begin{equation}
(s^2 -1)  \mu_f+ \rho_D 
\label{mub}
 \end{equation}  for $\rho_D:$ the rank of cycles (\ref{tangencycycles}).

Here we remark that 
 \begin{equation}
 \mu = {\rm rank} \;H_1((F^{\ast}(f))^{-1}(b),\mathbb{Z})  =  s^2  \mu_f + \rho_D
\label{mu}
 \end{equation}
thanks to \cite[Theorem 4.9]{YZ}.

From the definition of the morphism $F_*$ (\ref{Fast}) we see that $rank (ker(F_*)  ) = \mu - \mu_f.$


The combination of (\ref{mub}), (\ref{mu})  shows the equality  $rank(H) = rank(ker(F_*)).$  Thus together with  $H\subseteq ker(F_*)$ we conclude $H=ker(F_*).$

\end{proof} 

\section{Relatively Exact 1-forms}\label{REx}

A foliation $\mathcal{F} = \mF(\omega)$ defined by a holomorphic 1-form $\omega$ is called \emph{integrable} if there exists a meromorphic function $f$ on $\mathbb{P}^2$ such that 
 $df \wedge \omega|_{\mF}=0$. 
In this case the meromorphic function $f$ is said to be the \emph{first integral} of $\mathcal{F}$.

The concept of relatively exact forms has been investigated by many authors, e.g. \cite{Mu}, \cite[Section 4]{Mo1}.

\begin{defi}
A meromorphic 1-form $\omega$ on $\mathbb{P}^2$ is called relatively exact modulo a foliation $\mathcal{F}$  in $\mathbb{P}^2$  if the restriction of $\omega$ to each leaf $\mathcal{L}$ of $\mathcal{F}$ is exact, i.e. there is a meromorphic function $g$ on $\mathcal{L}$ such that $\omega|\mathcal{L}=dg$. 
\end{defi}

Let us call $f$- fiber the set $\{ u \in \proj \setminus {\mathcal R}: f(u) =t \}$ defined for some $t \in \C $. See Conditions \ref{cond}.

\begin{prop}\cite[\S 2]{Mu} \label{3:53}
A 1-form $\omega$ is relatively exact modulo the foliation $\mathcal{F}(df)$ with rational first integral $f$ if and only if 
\[
\int_{\delta} \omega=0
\] 
for every closed curve $\delta$ in a  $f-$fiber. 
\end{prop}
\begin{proof}
Let $L$ be a line in $\mathbb{P}^2$ which is not $\mathcal{F}$-invariant and does not pass through the point in $\mR$ as in Conditions \ref{cond}.
 For any point $u\in U:=\proj \setminus \mathcal{R}$  let 
 \[
f^{-1}(f(u))\cap L=\{p_1,p_2,\cdots,p_{r}\} 
 \]
 where the intersection multiplicity of $p_i$ might be greater than $1$. 

Define 
\begin{equation}\label{4:20}
\begin{split}
g&:\proj \setminus  \mR \to \mathbb{C}\\
& g(u)=\frac{1}{r}(\sum_{i=1}^r\int_{u}^{p_i}\omega)
\end{split}
\end{equation}
where $\int_{u}^{p_i}$ is an integral over a path in $f^{-1}(f(u))$ which connects $u$ to $p_i$. The function $g$ is well-defined and does not depend on the choice of the paths connecting $u$ to $p_i$ because 
$\int_\delta\omega=0$ on any close curve $\delta$ in each level set $f^{-1}(f(u))$.  Furthermore it is clear that a monodromy action on the line $L$ leaves the set  $L\cap f^{-1}(f(u))$ invariant as it induces merely a permutation among its points. Thus we conclude that $g(u)$ is a meromorphic function on $\proj \setminus (V(P) \cup V(Q))$ in taking Levi extension theorem and Hartogs theorem into account. 
\end{proof}
 
A function $f$ is called {\it non-composite} if every generic $f$-fiber is irreducible.  It is easy to see that $f$ is non-composite if and only if $f$ can not be factored as a composite
\begin{equation}
\mathbb{P}^2\overset{f'}\rightarrow\bar{\mathbb{C}}\overset{i}\rightarrow\bar{\mathbb{C}}
\label{composite}
\end{equation}
where $i$ is a non-constant holomorphic map.
In fact, if the composite factorization like (\ref{composite}) does not take place then the  generic $f$- fiber cannot be  reducible. From (\ref{composite})  the reducibility of generic $f$-fiber follows.

 Let $f=\frac{P^q}{Q^p}$ be a rational  function satisfying Conditions \ref{cond}
and suppose that for every $t\in\mathbb{C}$ the fiber $f^{-1}(t)$ is connected.  Let $\mathcal{F}(\omega_0)$  be a foliation on $\mathbb{P}^2$  with the non-composite  first integral $f$ not satisfying \eqref{composite} for 
\begin{equation} 
\omega_0=\frac{df}{f}.
\label{omega0}
\end{equation}
 Let $\omega$ be a rational 1-form with the pole divisor 
\begin{equation} 
\tilde D=n_1D_1+n_2D_2,\  \  \  \ where\  \  \  D_1:=V(P), D_2:=V(Q).
\label{Ddivisor}
\end{equation}

\begin{theo}\label{4:28:05}

  Every relatively exact rational  1-form $\omega$ modulo $\mathcal{F}(\omega_0)$ with pole divisor \eqref{Ddivisor} $\tilde D$ has the form 
\begin{equation}
\omega=dg+T\omega_0
\end{equation}
where $g$ and $T$ are rational  functions with the pole divisor $\tilde D$.
\end{theo}
The following is a modification of $\cite[Theorem \ 4.1]{Mo1}$  adapted to our situation.
 \begin{proof}
 The  function $g $ in \eqref{4:20} is a holomorphic function in $\mathbb{P}^2\setminus (D_1\cup D_2)$. For a point $u\in U\setminus (D_1\cup D_2)$, by the hypothesis  $q,p$ are the multiplicities of $f$ along $D_1, D_2$, respectively. The function $f^{\frac{n_1}{q}}$ is an  univalent function in a small neighborhood of the path connecting $u$ to $p_i$ and we have 
\[\int_{u}^{p_i}\omega=f^{(-\frac{n_1}{q})}\int_{u}^{p_i}f^{(\frac{n_1}{q})}\omega\]
 $f^{\frac{n_1}{q}}\omega$ is a holomorphic 1-form along $U\cap D_1$ therefore the above integral  has poles of order at most $n_1$ along $D_1$. 
 By using the chart around infinity and applying the above argument to $D_2$ once again, one can check that each component integral in \eqref{4:20} has poles of order
 at most $n_2$ along $D_2$.  
The equalities 
\[dg\wedge \omega_0=\omega\wedge \omega_0 \Rightarrow (\omega-dg)\wedge\omega_0=0\]
imply that there is a rational  function $T$ with pole divisor $\tilde D$ such that $\omega=dg+T\omega_0$.
 \end{proof}

\begin{coro}\label{22:12}
Suppose that $\omega$ is a polynomial homogeneous $1-form$  on $\mathbb{P}^2$ with $\deg(\mathcal{F}(\omega))= deg(\mathcal{F}(\omega_0))$ and $\frac{\omega}{F^{\ast}(P Q)}$ is relatively exact modulo $\mathcal{F}(\omega_0)$. Then there are polynomials $(P_1,Q_1)\in\mathcal{P}_{ms}\times\mathcal{P}_{ns}$ such that $\omega$ has the form 
\begin{equation}\label{22:20}
\omega=qF^{\ast}(Q)dP_1-pP_1dF^{\ast}(Q)-qQ_1dF^{\ast}(P) + pF^{\ast}(P)dQ_1.
\end{equation}
\end{coro}
\begin{proof}
By the Theorem \ref{4:28:05} there are polynomials $B,A$ of  degree at most $(m+n)s$ such that 
\[
\frac{\omega}{PQ}=d(\frac{B}{PQ})-(\frac{A}{PQ})(\frac{q.QdP-p.PdQ}{PQ})
\]
\begin{equation}\label{18:06:20:1}
\omega=\frac{QP.dB-Bd(PQ)-A(qQdP-p.PdQ)}{PQ}
\end{equation}
This implies that 
\[
P|B+qA, \ \ \ \ \ \ \ \ \ Q|B-pA\Rightarrow
\]
\[
B+qA=(p+q)P.Q_1,\ \ \ \ \ \ \ B-pA=(p+q)QP_1\Rightarrow
\]
\[
B=p PQ_1+qQP_1,\ \ \ \ \ \ A=-QP_1+ P.Q_1
\]
where $P_1,Q_1$ are two polynomials of respective degrees at most $ms,ns.$ 
Substituting these  in (\ref{18:06:20:1}) we get the result.
\end{proof}

In the sequel, for a form $\omega$
defined on $\proj$  we shall use the notation $\omega|$ of its restriction on the affine variety $\proj$.

\begin{coro} 
The morphism $F^*:H_f\to H_{F^{\ast}(f)}$ is injective and the image of the basis of $H_f$ is can be  extended to a basis of $H_{F^{\ast}(f)}$.
\label{Hfinj}
\end{coro}

\begin{proof} We consider the pull-back by $F$ of the projectivized 1-form ${\omega}$  for  $\omega| \in H_f.$
We restrict $F^*({\omega})$ on $\mathbb{C}^2=\{z=1\}$ in $H_{F^{\ast}(f)}$. It is well known that $F^*$ is injective (see \cite[Proposition 1.1]{Hart}). For each element $\tilde\omega_j|=m_j \eta| $ of the basis $H_f$, ${\tilde \omega_j}={m_j}(X,Y,Z){\eta}$ is a polynomial 1-form on $\proj$ and $F^*({\tilde \omega_j})|={m}_j (R,S,T)F^*({\eta})|$ where the form $\eta|=axdy-bydx$ such that  $d \eta| =  dx \wedge dy$.  According to Theorem \ref{14:06}, $F^*({\eta})|$ can be written as
\[F^*({\eta})|=\sum_{\ell =1}^\mu g_\ell (F^{\ast}(f) ) {\tilde \omega_\ell}| ,\]
where ${\tilde \omega_\ell}| $ is a basis of $H_{F^{\ast}(f)}.$ The inequality (\ref{degak}) entails
$$deg(g_\ell)\leq \frac{deg(F^*({\eta}))+n s-deg(Z({\tilde \omega_\ell}))-1}{ms.q}<1.$$ 
We recall  here that $ deg(F^*({\eta})) \leq s $ and $m > n \geq 2 $ by Condition \ref{cond},(3).

This means that the polynomial $g_\ell$ is in fact a constant for every $\ell$. Therefore  the 1-form
 $$ F^*({\tilde \omega_j})|= {m}_j(R,S,T)F^*({\eta})|$$ is free of $F^{\ast}(f)$. This terminates the proof.
\end{proof}

\section{Tangent vector}\label{TC}

We begin our discussion on the tangent vector of 
$\mathcal{M}(2,d)$  ($d=s(a+2)-2$) at the point $\mathcal F(F^*(\alpha_0)) \in \mathcal{I}(ms-1,ns-1)\cap\mathcal{P}(2,a,s)$, for $\alpha_0$ defined in \eqref{alpha0}.
First of all, we show the following lemma on a decomposition (\ref{19:31}) valid for
$\omega_1$  used to define the first Melnikov function \eqref{Mk}.

\begin{lem}\label{19:03} For $\omega_1$ in \eqref{Mk} we find
a homogeneous 1-form $\alpha$ on $\mathbb{P}^2$ of degree $a+1=m+n-1$ and two homogeneous polynomials $P_1,Q_1$ of respective degrees $ms,\ ns$ such that  
\begin{equation}\label{19:31}
 \omega_1=F^*( {\alpha})+ \omega_e, 
\end{equation}
where 
\begin{equation}\label{omegae}
 \omega_e=qF^{\ast}(Q)dP_1-pP_1dF^{\ast}(Q)-qQ_1dF^{\ast}(P) +  pF^{\ast}(P)dQ_1.
\end{equation}
\end{lem}

\begin{proof}
In the affine coordinate the polynomial map $F$ introduces a morphism $F^*:H_f\rightarrow H_{F^{\ast}(f)}$
between two $\mathbb{C}[\tau]$-module $H_{f}$ and $\mathbb{C}[t]$-module $H_{F^{\ast}(f)}$.
The linear map 
\[
H_1(f^{-1}(b),\mathbb{Z})\rightarrow\mathbb{C} \ given \ by \ \ \Delta\rightarrow
\int_{F_*^{-1}(\Delta)}\frac{\omega_1}{F^{\ast}(P Q)}
\]
is well-defined  because $\int_\delta \frac{\omega_1}{F^{\ast}(P Q)}=0$, $\forall \delta \in {ker(F_*)}$ by virtue of Theorem \ref{thmFast}.

By the duality between de Rham cohomology and singular homology there is a $C^\infty$ differential form 
$\alpha_b$ in regular fiber $f^{-1}(b)$ such that 
\[
\int_{F_*^{-1}(\Delta)}\frac{\omega_1}{F^{\ast}(PQ)}=\int_{\Delta}\frac{\alpha_b}{PQ}.
\]
According to Atiyah-Hodge  theorem  (See \cite[Theorem 4]{AH}, \cite[Chapter 4]{M-VL}) $\alpha_b$ can be taken holomorphic thus polynomial.
The analytic continuation of  $\alpha_b$ with respect to the parameter $b \in \C \setminus C_f$ gives rise to a holomorphic global section $\alpha$ of cohomology bundle of $f$.
Thus in the affine coordinate $\mathbb{C}^2 \subset \proj$ we  have the following decomposition in $H_f$ 
\begin{equation}\label{ahf}
\alpha|=\sum_{ \ell=1}^{\mu_f} h_{\ell} (f|)\eta_{\ell}| 
\end{equation} 
where $h_{\ell}(\tau)$ is holomorphic  in $ \tau \in \mathbb{C}\setminus C_f.$

The coefficients $h_{\ell}(\tau)$ in \eqref{ahf}  are rational functions in $\tau$ because of the following relation
\[
\begin{bmatrix}
  h_{1} (\tau)\\
   \vdots  \\
   h_{{\mu_f}}(\tau)
\end{bmatrix}
=
\begin{bmatrix}
  \int_{\delta_k}\eta_{\ell}| 
\end{bmatrix}_{\mu_f\times\mu_f}^{-1}
\begin{bmatrix}
  \int_{\delta_1}\alpha| \\
   \vdots  \\
    \int_{\delta_{\mu_f}}\alpha|
\end{bmatrix}
\]
All the elements of the matrices in the right side of the equality have finite growth at critical values. This is an analogy of the argument used to show \eqref{23:28} with the aid of Cramer's rule.

Pull-back of forms $\eta_{\ell}|, \forall \ell$ are independent in $H_f$ under the map $F^*$ and can be extended to a basis for $H_{F^{\ast}(f)}$ in view of Corollary  \ref{Hfinj}.

There is a polynomial $K(\tau)\in \mathbb{C}[\tau]$ such that $K(f|).\alpha|$ be a holomorphic form. We can write $K(f|).\alpha| =\sum_{\ell} {h'}_{\ell}(f|)\eta_{\ell}|$ 
then $F^*(K)\omega_1|-F^*(K.\alpha|)=0$ in $H_{F^{\ast}(f)}$.

Now we shall show the Claim: $\omega_1|-F^*(\alpha|)=0$ in $H_{F^{\ast}(f)}.$ 
The set  $\{ F^*(\eta_{\ell}|) \}_{\ell=1}^{\mu_f}$  can be extended to a basis of $H_{F^{\ast}(f)}$ by  Corollary  \ref{Hfinj}.
So we have in $H_{F^{\ast}(f)}$
\begin{eqnarray}\label{7:26}
F^\ast(K)\omega_1|=\sum_{\ell=1}^{\mu_f}F^\ast(K). {h'}_{\ell}(F^{\ast}(f|))F^*(\eta_{\ell}|) +\sum_{\sigma= \mu_f+1}^\mu F^*(K)a_{\sigma}\tilde{\eta}_{\sigma}|
\end{eqnarray}
Here  $\{\tilde{\eta}_{\sigma}\}_{\sigma= \mu_f+1}^\mu$ is a basis of  $H_{F^{\ast}(f)}$ alien to  $F^\ast(H_f).$

Since each element of $H_{F^{\ast}(f)}$ can be uniquely written as a linear combination of the elements in this basis we get the vanishing coefficients $a_{\sigma}=0$ for all $\sigma$. In other words, in view of \eqref{ahf}, we have  $F^*(K). {h}_{\ell}=F^*( {h'}_{\ell})$ hence $K|{h'}_{\ell}$.  This means that $\omega_1|-F^*(\alpha|)=0$ in $H_{F^{\ast}(f)}$. 
 This is nothing but the Claim in question.

To find the degree of $\alpha$, we write  \[\omega_1=\sum_{l}F^{*}( h_l \eta_l)=\sum_{\ell} h_l(F(f))\sum_{\beta}g_{\beta}^{l}(F(f))\eta_\beta
=\sum_{\ell, \beta }F^\ast(g_{\beta}^{l}h_l)\eta_\beta,\] 
and we conclude that  $deg(h_l)=0$ by virtue of  Theorem \ref{14:06}. Therefore $deg(\alpha)=a+1$.

Thanks to the Claim, we have 
\begin{equation}\label{omegakFalpha}
\int_{\delta} \frac{\omega_1-F^*(\alpha)}{F^{\ast}(P Q)}=0, \ \ \ \forall \ \delta\in H_1((F^{\ast}(f))^{-1}(b),\mathbb{Z}),
\end{equation}
which implies that the integrand rational form of \eqref{omegakFalpha}  is a relatively exact 1-form modulo the foliation $F^*(\omega_0)$ for \eqref{omega0}.  By Corollary \ref{22:12} there is a 1-form $\omega_e$ of the form (\ref{22:20}) such that $\omega_1=F^*(\alpha)+\omega_e. $ In fact there are polynomials $P_1$ and $Q_1$ with degree $ms$ and $ns$ respectively such that $ \omega_e=qF^{\ast}(Q)dP_1-pP_1dF^{\ast}(Q)- qQ_1dF^{\ast}(P)+pF^{\ast}(P)dQ_1$.
\end{proof} 
We know that there are rational function $ \tilde h_1$ and a $1-form$  $\beta_1$ on $\mathbb{P}^2$ such that  
\begin{equation}\label{omegaePQ}
\frac{{\omega}_e}{ F^{\ast}(P Q)}=\frac{\beta_1 +\tilde h_1 F^\ast (\alpha_0)}{F^{\ast}(P Q)  }.
\end{equation}

\begin{lem}\label{correctionterm}
The 1-form $ \omega_e$ in the equality (\ref{19:31}) is  of the form

\begin{equation}\label{19:45}
 \omega_e=qF^{\ast}(Q)dP_1-pP_1dF^{\ast}(Q)- qQ_1dF^{\ast}(P) + pF^{\ast}(P)dQ_1
\end{equation}
with $P_1 =q <F_1 , F^\ast (grad P) >, $ $ Q_1= p<F_1 , F^\ast (grad Q) >$
for a vector 
\begin{equation}\label{vecV1} 
F_1 = (R_1, S_1, T_1),
\end{equation}
defined by some  homogeneous polynomials $R_1,S_1,T_1 \in \C[x,y,z]^h_s.$ 
\end{lem}
\begin{proof}
First of all we introduce the following polynomials $\{\lambda_j (X,Y,Z)\}_{j=1}^3$
defined by the relation 
\begin{equation}\label{lambda123}
\alpha_0 =  qQdP -pP dQ  =  \lambda_1 dX + \lambda_2 dY + \lambda_3 dZ.
\end{equation}

Secondly, we define polynomials $\{\rho_j (x,y,z)\}_{j=1}^3$ by means of \eqref{JF}, \eqref{lambda123},
\begin{equation}\label{rho123}
(F^\ast(\lambda_1), F^\ast(\lambda_2), F^\ast(\lambda_3)).J_F(x,y,z)= ( \rho_1, \rho_2, \rho_3).
\end{equation}
In other words, $ F^\ast(\alpha_0) = \rho_1 dx +  \rho_2 dy +  \rho_3 dz. $

Now we pass to the investigation of polynomials $P_1, Q_1$ present in \eqref{omegae}.
By multiply the relation \eqref{omegaePQ} with $F^*(\alpha_0)$ we get the following equality: 
\begin{equation}\label{P1Q1Eqn}
\begin{split}
(-Q_1F^{\ast}(P)+P_1F^{\ast}(Q) ) F^\ast (dP\wedge dQ)= 
(pq)^{-1}\left( \beta_1-qF^{*}(Q)dP_1+pF^{*}(P)dQ_1\right)\wedge F^{*}(\alpha_0).
\end{split}
\end{equation}

Now let us consider the ideals in $\C[x,y,z]$
\[   I_1=<F^\ast(\lambda_1), F^\ast(\lambda_2), F^\ast(\lambda_3)>, \;\;\;\;   I = <  \rho_1, \rho_2, \rho_3> \] 
 for polynomials from \eqref{lambda123}, \eqref{rho123}.
 We also consider the ideal $J=<J_F(j,\ell)>_{1 \leq j, \ell \leq 3}\subset\mathbb{C}[x,y,z]$ generated by  $2\times 2$ minors of 
$J_F(x,y,z)$ \eqref{JF}.

Since $V(J)\cap V(I_1)=\emptyset$ we see that  $I_1+J=\mathbb{C}[x,y,z]$ and $I_1\cap J= I_1.J.$ 

The equality  \eqref{P1Q1Eqn} entails that $ (-Q_1F^{\ast}(P)+P_1F^{\ast}(Q)) . J \in I \subset I_1.$ This means that $  (-Q_1F^{\ast}(P)+P_1F^{\ast}(Q)) . J \in I_1 \cap J = I_1.J$ thus $ (-Q_1F^{\ast}(P)+P_1F^{\ast}(Q)) \in I_1. $
In other words, there exist polynomials $R_1,S_1,T_1$ of degree $s$ such that 
\begin{equation}\label{F1R1}
\begin{split}
-Q_1F^{\ast}(P)+P_1F^{\ast}(Q)&=R_1 F^{\ast}(\lambda_1) + S_1 F^{\ast}(\lambda_2) +T_1 F^{\ast}(\lambda_3).
\end{split}
\end{equation}
Since $P,Q$ are co-prime we have the required expressions for $P_1, Q_1$ \eqref{19:45}:

\begin{equation}\label{P1Q1}
\begin{split}
P_1 &= q(R_1.F^\ast(P_X)+S_1.F^\ast(P_Y)+T_1.F^\ast(P_Z)),\\ 
Q_1&= p(R_1.F^\ast(Q_X)+S_1.F^\ast(Q_Y)+T_1.F^\ast(Q_Z)).
\end{split}
\end{equation} 
\end{proof}

Now we proceed to the proof of Theorem \ref{8:34:21}.
\begin{proof}
 Let us introduce the notation $P_0=qP, \ Q_0=pQ.$ 
We see that the 1-form $F^{*}(qQ_0dP_0-pP_0dQ_0)$ defines the foliation $\mathcal{F}(\alpha_0).$  With this notation and \eqref{10:16}, \eqref{P1Q1}  we have
$ P_1  = <F_1, F^\ast(grad \; P_0)>,   Q_1  = <F_1, F^\ast(grad \; Q_0)>$.
 
The $\epsilon^1$ part of the numerator of the rational form
$$d\left(\frac{(P_0+ \epsilon P_1)^{q}}{(Q_0-\epsilon Q_1)^{p}}\right )$$
gives rise to 
$$ pq (q Q_0dP_1 - pP_1 dQ_0 - qQ_1 dP_0 +  pP_0d Q_1 ) =$$
\begin{equation}\label{3:32}
\begin{split}
 qQ d( <F_1, F^\ast(grad \; P)>) &- p  <F_1, F^\ast(grad \; P)> dQ\\
 - q <F_1, F^\ast(grad \; Q)> dP &+ p P d(<F_1, F^\ast(grad \; Q)>). 
\end{split}
\end{equation} 

If we consider the 1-form \eqref{19:45} in replacing $(P,Q)$ by $(P_0,Q_0)$  and divide it by $pq$, then the result will concide with
 \eqref{3:32}. This implies that every tangent vector $\omega_1$   to ${\mathcal  X}(a,s)$ at $\mF_0$  from \eqref{11:31} can be interpreted as a tangent vector $\omega_W$ \eqref{omegaW} to  ${\mathcal P}(2,a,s)$ at the same point. 
\end{proof}



 \section{Brieskorn/Petrov  module}\label{PMBM}

In this section we establish a new formulation of results concerning the Brieskorn/Petrov module
defined for a rational function of type \eqref{PQfirstintegral}. This generalization furnishes us with the proof of Theorem
\ref{14:06} that is necessary to establish Corollary \ref{Hfinj} that in its turn  plays an essential r\^ole during the proof of  the key Lemma  \ref{19:03}.

Let us consider a rational  function $f:\mathbb{P}^2\setminus\mathcal{R}\to \mathbb{C}$ as in (\ref{PQfirstintegral}) satisfying Conditions \ref{cond}. Further in this section, we regard   $f=\frac{P^q}{Q^p}$ with $P, Q \in\mathbb{C}[X,Y]$
defined on $\C^2$. In other words in the sequel $P=P(X,Y,1)$, $Q =Q(X,Y,1)$ in terms of polynomials in (\ref{PQfirstintegral}). This reduction is possible due to the fact that $f$ is transversal to $Z=0$.
We recall that $D:=f^{-1}(\infty) = V(Q)$ is smooth due to the Condition \ref{cond}, (1). Let  $\Omega^i(*D)$ be the set of rational  $i$-forms on $\mathbb{P}^2$ with poles of arbitrary order along $D$.  Let $t$ be an affine coordinate of $\mathbb{C}=\mathbb{P}^1\setminus\{\infty\}$. The set $\Omega^i(*D)$ can be regarded as a  $\mathbb{C}[t]$-module according to  the following identification: 
$$p(t).\omega=p(f)\omega, \ \ \ \omega\in \Omega^i(*D) $$
 for $p(t)\in \mathbb{C}[t].$
Any $i$-form $\omega\in\Omega^i(*D)$ can be considered as a polynomial $i$-form in three  variables $X,Y,\zeta$ 
with $d\zeta=-\zeta^2dQ$. 
 The $\mathbb{C}[t]$-module of relative rational  2-forms with poles of arbitrary order is defined as follows:
\[\Omega_{\mathbb{P}^2/\mathbb{P}^1}^2(*D)=\frac{\Omega^2(*D)}{df\wedge\Omega^1(*D)}.\]
 One can find a $\mathbb{C}[t]$-module injective and surjective homomorphism from $\Omega_{\mathbb{P}^2/\mathbb{P}^1}^2(*D)|_{\C^2}$  to the following
  $\mathbb{C}[t]$-module with quotient ring structure
\begin{equation}\label{15:39}
\mM(*D):= \frac{\mathbb{C}[X,Y,\zeta]}{I}
\end{equation}
defined for the ideal
\begin{equation}\label{15:391}
I = <\ \zeta.Q-1 ,\ qP_X-p \zeta.Q_X.P,\ qP_Y-p\zeta.Q_Y.P>.
\end{equation}
We remark here that  the element $Q$ is invertible in the quotient ring $\mM(*D).$ 

\begin{prop} $\Omega_{\mathbb{P}^2/\mathbb{P}^1}^2(*D)$ has a structure of vector space of dimension $\mu_f=(n+m-1)^2-nm$ where $\mu_f$ is the global Milnor number of $f$. 
\label{globalMilnor}
\end{prop}
\begin{proof}
According to \cite{Br-inv}, \cite[Corollary 1.1]{Mo3},  $\Omega_{\mathbb{P}^2/\mathbb{P}^1}^2(*D)$ is a vector space with dimension $\mu_f:$ the sum of local Milnor numbers of $f$.
We remark here that the cardinality of the set  ${\mathcal R} = V(P) \cap V(Q)$  is equal to $nm.$
\end{proof}

We can prove the Proposition \ref{globalMilnor} by the aid of the Gr\"obner basis.  In fact by using the graded lex order on $\mathbb{C}[X,Y,\zeta]$ one can find a Gr\"obner basis $\tilde I$ for the ideal $I$ \eqref{15:391} and then by considering the leading part of  $\tilde I$ the basis of  $\mM(*D)$ is obtained. Indeed this basis depends on $\zeta$ but by changing the basis we can write $\zeta$ as a polynomial in variables $X,Y$ because $\zeta$ is an invertible. 

By the Condition \ref{cond}, (1)  
we have
\begin{equation}
\mM(*D) \cong \sum_{j=1}^{\mu_f}  \C. m_j(X,Y) 
\label{mk}
\end{equation}
for $m_j(X,Y) \in \C[X,Y].$ 

We see that the concrete calculation of the monomials  $m_j(X,Y)$ in \eqref{mk} can be done with the following  example that essentially covers all necessary cases \eqref{PQfirstintegral} under the Conditions \ref{cond}.  Namely these conditions mean that the Newton polyhedron of
the polynomial $P(X,Y)$ (resp. $Q(X,Y)$) is a triangle with vertices $\{ (0,0), (m,0), (0,m) \}$ (resp.  $\{ (0,0), (n,0), (0,n) \}$ ) with non-degenerate condition on the edge $s (m,0)+(1-s) (0,m), s \in [0,1] $ (resp. $  s (n,0)+(1-s) (0,n), s \in [0,1] $).

Consider two generic  polynomials of respective  degrees $n,m$ of the form $P=p_1X^m+p_2Y^m$ and $Q=q_1X^n+q_2Y^n$ where $P,Q$ are co-prime. The module  $\mM(*D)$ is generated by
\[
X^iY^\ell, \  \  \  s.t \   \  \  \  X^{m-1} Y^{n-1} \not| X^iY^\ell   \ , \  \  \  \  0\leq i,\ell \leq (m+n-2),
\]
\begin{figure}[htp]
\centering
\includegraphics[width=2.7in]{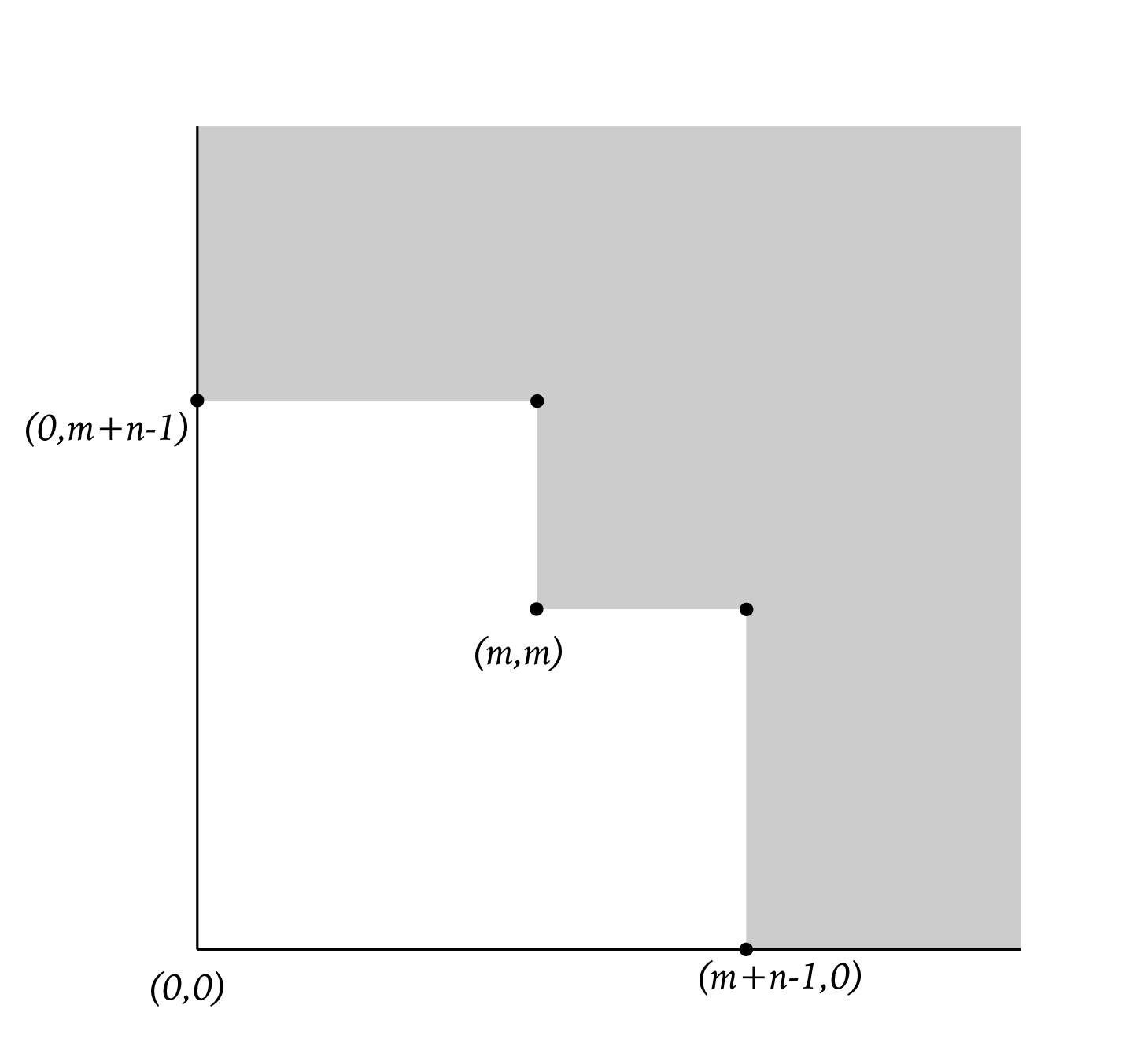}
\caption{$(i,j)\leftrightarrow X^iY^\ell$}
\label{mbasis}
\end{figure}

 Thus we get monomials $m_j(X,Y)$ for $j = 1, \cdots, \mu_f=(n+m-1)^2-nm.$ 

See Figure \ref{mbasis}.

Now let us define the relative cohomology associated to $f$ that is endowed with $\mathbb{C}[t]$-module structure 
\begin{equation}
H_f:=\frac{\Omega^1(*D)}{df\wedge\Omega^0(*D)+d\Omega^0(*D)}.
\label{BPmodule}
\end{equation}
where $t=f$ in the affine coordinate $\mathbb{C}$.
This module is called {\it Brieskorn $\mathbb{C}[t]$-module} (or {\it Petrov module} according to \cite{L.G}).
It is worthy noticing that the denominator of (\ref{BPmodule}) represents the space of relatively exact rational 1-forms modulo
 $\mathcal{F}(\omega_0)$ for \eqref{omega0}.

We define the degree of the zero divisor $Z(\eta)$ of a rational form $$ \eta = \frac{P^1_\eta (X,Y) dX + P^2_\eta(X,Y) dY }{Q^\ell}, \ell \geq 0$$ for $Q \not|  P^1_\eta,  P^2_\eta$ as follows
\begin{equation}
deg Z(\eta) = max\{deg P^1_\eta (X,Y),deg P^2_\eta(X,Y)\}
\label{degZ}
\end{equation}
 
 Now we state the following (see \cite[Theorem 10.12.1]{Mo6}):

\begin{theo}\label{14:06}
The $\mathbb{C}[t]$-module $H_f$ is free and finitely generated by  1-forms $\alpha_j$, $j=1, \cdots,  \mu_f$ and each 1-form $\alpha_j$ can be defined by the condition 
\[d \alpha_j=m_j(x,y) dx\wedge dy,\]
where $m_j(x,y)$ is an element of the monomial basis (\ref{mk}) for $\mM(*D)$. Furthermore, a rational  1-form $\alpha$ in $\Omega^1(*D)$ can be written as follows
\[
\alpha=\sum_{j=1}^\mu {\mathcal C}_j(f)\alpha_j+df\wedge\zeta_1+d\zeta_2
\]
where $\zeta_1,\ \zeta_2\in\Omega^0(*D)$ and $\mC_j$ is a polynomial with degree that admits an evaluation,
\begin{equation}
deg (\mC_j)  \leq \frac{deg(Z(\alpha))+ deg(Q)-deg(Z(\alpha_j))-1}{q deg(P)} = \frac{deg(Z(\alpha))+n-deg(Z(\alpha_j))-1}{mq}.
\label{degak}
\end{equation}
Here $Z(\eta)$ is zero divisor of $\eta$.

\end{theo}

In order to give a proof of this theorem, we introduce period  integrals associated to the rational function $f$ as in \eqref{PQfirstintegral}.

Let $\delta_i(t)$ , $i=1, \cdots, \mu_f$ be a continuous family of vanishing cycles around the critical points of $f$ which form a basis of $H_1(f^{-1}(t),\mathbb{Z})$ for any regular value $t\in\mathbb{C}$. 
Suppose that the 1-forms $\alpha_1,\cdots,\alpha_{\mu_f} \in H_f$
  such that $d\alpha_j=m_{j}dx\wedge dy$ where $m_{j}$ is an element of the basis (\ref{mk}) for $\mM(*D).$
With these homology and cohomology bases one can associate the \emph{period matrix } 
\begin{equation}\label{Yt}
{\bf Y}(t):=
\left(\begin{array}{lll}
\int_{\delta_1(t)}\alpha_1&\cdots&\int_{\delta_\mu(t)}\alpha_1\\
\vdots&\ddots&\vdots\\
\int_{\delta_1(t)}\alpha_\mu&\cdots&\int_{\delta_\mu(t)}\alpha_\mu
\end{array}\right)
\end{equation}
which is an analytic multi-valued matrix-function ramified over the critical values $C_f$ of $f$.
Also, $W(t):=det({\bf Y}(t))$ is called  Wronskian function. 
\begin{prop}
For $f$ under the Conditions \ref{cond},
 the determinant $W(t)$ of any period matrix is a polynomial in $t$ with zeros at $C_f.$ This fact holds regardless of the choice of the forms which constitute a  basis of $H_f.$
\end{prop}
\begin{proof}
From the Picard-Lefschetz theorem we see that the monodromy action around each $t = t_i \in C_f$ induces
a monodromy transformation $T_i {\bf Y}(t)$ on ${\bf Y}(t)$ (\ref{Yt}) and $det (T_i {\bf Y}(t)) = det ({\bf Y}(t)). $ This means that
the determinant  $W(t)$ is a single-valued function on $\mathbb{C}$.
As $t$ tends to infinity the integrals occurring in the entries of ${\bf Y}(t)$ grow no faster than a  polynomial in $|t|$ in any sector. This implies that $W(t)$ is a polynomial. When $t$ tends to a point $c\in C_f$, at least one of vanishing cycles $\delta_i(t)$ vanishes, hence one row of ${\bf Y}(t)$ is zero and the determinant of $W(t)$ tends to zero. 
\end{proof}

Let us define the function 
\begin{equation}
\Delta(t)=(t-t_1)^{\mu_1}(t-t_2)^{\mu_2}\cdots(t-t_s)^{\mu_s}
\end{equation}
where $\mu_i$ is the summation of  local Milnor numbers of  critical points  located on the fiber $f^{-1}(t_i)$ and $\sum_{i=1}^s\mu_i=\mu$. When the $f$ is a polynomial $\Delta$ is called determinant function of $f$ (see e.g. \cite{L.G}).
\begin{lem}\label{14.01}
There exists a non-zero constant $c$ such that $W(t)=c\Delta(t)$.
\end{lem}  
\begin{proof}
If $t$ tends to the critical value $t_i$ then the vanishing cycles $\delta_i(t)$ correspond to $t_i$ tend  to  points therefore $\int_{\delta_i(t)}\alpha_j$ tends to zero. This implies that the matrix ${\bf Y}(t)$ at point $t_i$ is not of full rank, in other words $t_i$ is a root of Wronskian function $W$. 
If all critical values $C_f$ of $f$ are distinct  then we $W(t)=c\cdot\Delta(t)$ as the cardinality of $C_f$ is $\mu_f$.
 
If not, consider a deformation $$f_{A,t}(X,Y)=\frac{(P+a_1 X+b_1Y)^q}{(Q+a_2 X+b_2 Y)^p}-t$$ where 
$A=\left( a_1, b_1,  a_2, b_2 \right)$.
There is an open subset $U$ of $ (A, t) \in \mathbb{C}^5$ 
 such that the function $f_{A,t}$ has distinct critical values and all the critical points belongs to affine coordinate $\mathbb{C}^2$. Let $\Sigma_{A,t}=\{(A,t)\ |\ D(A,t)=0\}$. As the Milnor number of $f_{a,t}$ is equal to the Milnor number $f$ then $D(A,t)$ is also has degree $\mu$ therefore if $A=0$ then $D(0,t)=\Delta(t)$. 
Let $\{\delta_j(t,A)\}_{j=1}^{\mu_f} 
$ be a continuous family of cycles which forms a basis of $H_1(f_{A,t}^{-1}(0),\mathbb{Z})$ for any $(A,t)\notin\Sigma_{A,t}$. For the polynomial $f_{A,t}$ consider the Wronskian function 
\[
\tilde W(A,t)=det\left( \int_{\delta_i(t,A)}\alpha_j\right).
\]
This function is polynomial in $(A,t) $ and vanishes along $\Sigma_{A,t}$, this implies that $\tilde W(A,t)=C_A\cdot D(A,t)$. The polynomial $C_A$ does not depend on $t$ because $deg(\tilde W(A,t))$ respect to $t$ is equal to $deg(D(A,t))=\mu_f$. To finish the proof it is enough to recall that $ \tilde W(0,t)=W(t)$.
\end{proof}

Now we pass to the proof of Theorem \ref{14:06}. 

\begin{proof}

First of all we see that $\{\alpha_j\}_{j=1}^{\mu_f}$ are linearly independent in $H_f.$  The form $\frac{d\alpha_j}{df}$ of $\alpha_j$ coincides with the Gel'fand-Leray form of \[m_j\frac{dx\wedge dy}{df}.\]
According to  in \cite[Thoerem 5.25, b]{ZO} we have
\begin{equation}
det\left(\int_{\delta_i(t)}\frac{d\alpha_j}{df}\right)=det\left(\frac{d}{dt}\int_{\delta_i(t)}\alpha_j\right) = c=const\neq 0.
\end{equation}
This means the required linear independence. We remark also that this relation can be deduced from a refinement of \cite[Theorem 2.7]{Tan1} established for the Gauss-Manin system satisfied by period integrals $\int_{\delta_i(t)}\alpha_j$.

Let us consider a system of equations for a column vector of  unknown functions $ {\bf \mC} (t) = (\mC_1(t), \cdots, \mC_{\mu_f}(t))$
\begin{equation}
{\bf Y}(t).{\bf \mC} (t) = {\bf A_\delta}(t)
\label{YCDA}
\end{equation}
with $ {\bf Y}(t) $ \eqref{Yt} and a ${\mu_f}$ column vector:
\[ {\bf A_\delta}(t) =
\left(\begin{matrix}
\int_{\delta_1(t)}\alpha\\
\vdots\\
\int_{\delta_\mu(t)}\alpha
\end{matrix}\right).
\]   
Cramer's rule solves \eqref{YCDA} with solutions 
$${ \mC}_j (t) = det {\bf Y}_{{\bf A}, j}(t)/ det {\bf Y}(t) \;\;\; j=1, \cdots, \mu_f $$
where ${\bf Y}_{{\bf A}, j}(t)$ is a $\mu_f \times \mu_f$ matrix obtained by replacing the $j-$th column of ${\bf Y}(t)$ with  ${\bf A_\delta}(t)$.  The function $ det {\bf Y}_{{\bf A}, j}(t)$  has at most polynomial growth as $|t| \rightarrow \infty$ or as $t$ approaches to  a zero of $\Delta(t)$. By Lemma \ref{14.01}  $  \forall j, { \mC}_j (t) \in \mathbb{C}[t]$ as ${\bf Y}_{{\bf A}, j}(t)$  vanishes at zeros of $\Delta(t)$.

In other words, for a fixed $i$ we have 
\[\int_{\delta_i(t)}\alpha=\sum_{j=1}^\mu \mC_j(t)\int_{\delta_i(t)}\alpha_j.\]
This implies that 
\begin{equation}\label{23:28}
\alpha=\sum_{j=1}^\mu \mC_j(f)\alpha_j
\end{equation}
in the $\mathbb{C}[t]$-module $H_f$. Assume that $\alpha= Q^{-l} (P^1_\alpha dx+  P^2_\alpha dy)$ has poles of order $l$ along $V(Q)$ therefore $d\alpha={h}{Q^{-l-1}} dx \wedge dy$ where $h$ is a polynomial of degree $deg(Z(\alpha))+n-1$. 
 The derivation of each term of the right side of \eqref{23:28} becomes
$$ d(\mC_j(f)\alpha_j)  = \mC_j'(f) f \frac{( qQdP - pPdQ) \wedge \alpha_j}{PQ} + \mC_j(f) d \alpha_j.$$ 
Each term of the above expression has degree $( deg (\mC_j) -1) mq + (mq -1) + (n+m-1) +deg  Z (\alpha_j)$ and 
$mq .deg (\mC_j) +deg  Z (\alpha_j) -1$ respectively. We remark here that
$$ deg Z(( qQdP - pPdQ) \wedge (x dy - y dx)) =  n+m-1.$$

By Condition \ref{cond}, (3)  $ m \geq n \geq 2$ and  we conclude that 
\begin{equation} \label{17:06}
deg(Z(\alpha))+n-1\geq mq.deg(\mC_j)+deg(Z(\alpha_j)).
\end{equation}
\end{proof}

\smallskip
\begin{flushleft}
\begin{minipage}[t]{6.2cm}
  \begin{center}
{\footnotesize
{Y. Zare , S. Tanab\'e}\\
Department of Mathematics,\\
Galatasaray University,\\
\c{C}{\i}ra$\rm\breve{g}$an cad. 36,\\
Be\c{s}ikta\c{s}, Istanbul, 34357, Turkey.\\
{\it E-mails}:   {yadollah2806@gmail.com, yzare@gsu.edu.tr}\\
tanabe@gsu.edu.tr}
\end{center}
\end{minipage}
\end{flushleft}

\end{document}